\documentclass[platex]{amsart}
\usepackage{amssymb,amsmath,amsthm,empheq,comment,url,color}

\usepackage[top=30truemm,bottom=30truemm,left=20truemm,right=20truemm]{geometry}

\theoremstyle{definition}
\newtheorem{theorem}{Theorem}[section]
\newtheorem{proposition}[theorem]{Proposition}

\newtheorem{lemma}[theorem]{Lemma}
\newtheorem{corollary}[theorem]{Corollary}
\newtheorem{remark}[theorem]{Remark}
\newtheorem*{remark*}{Remark}

\DeclareMathOperator{\re}{Re}
\DeclareMathOperator{\im}{Im}

\DeclareMathOperator{\loc}{loc}
\DeclareMathOperator{\app}{app}

\DeclareMathOperator{\ModOp}{Mod_{op}}

\title[Minimal-mass blow-up solutions for inhomogeneous NLS with growth potentials]{Minimal-mass blow-up solutions for inhomogeneous nonlinear Schr\"{o}dinger equations with growth potentials}
\author[N. Matsui]{Naoki Matsui}
\date{\today}
\address[N. Mastui]{Department of Mathematics\\ Tokyo University of Science\\ 1-3 Kagurazaka, Shinjuku-ku, Tokyo 162-8601, Japan}
\email[N. Matsui]{1120703@ed.tus.ac.jp}

\keywords{blow-up rate, critical exponent, critical mass, growth potential, inhomogeneous, minimal-mass blow-up, nonlinear Schr\"{o}dinger equation.}
\subjclass[2010]{35Q55}

\providecommand{\keywords}[1]
{
  \small	
  \textbf{\textit{Keywords---}} #1
}

\begin{document}
\maketitle

\begin{abstract}
In this paper, we consider the following equation:
\[
i\frac{\partial u}{\partial t}+\Delta u+g(x)|u|^{\frac{4}{N}}u-Wu=0.
\]
We construct a critical-mass solution that blows up at a finite time and describe the behaviour of the solution in the neighbourhood of the blow-up time. Banica-Carles-Duyckaertz (2011) has shown the existence of a critical-mass blow-up solution under the assumptions that $N\leq 2$, that $g$ and $W$ are sufficiently smooth and that each derivative of these is bounded. In this paper, we show the existence of a critical-mass blow-up solution under weaker assumptions regarding smoothness and boundedness of $g$ and $W$. In particular, it includes the cases where $W$ is growth at spatial infinity or not Lipschitz continuous.
\end{abstract}

\section{Introduction}
We consider the following nonlinear Schr\"{o}dinger equation with potentials:
\begin{align}
\label{NLS}
\tag{NLS}
\left\{\begin{array}{l}
i\frac{\partial u}{\partial t}+\Delta u+g(x)|u|^{\frac{4}{N}}u-Wu=0,\\
u(t_0)=u_0
\end{array}\right.
\end{align}
in $\mathbb{R}^N$, where $g\in L^\infty(\mathbb{R}^N)$ and $W$ is the sum of potentials satisfying one of the following conditions:
\begin{align}
\label{W1}
\tag{W1}
W\in C^\infty(\mathbb{R}^N),\quad W\geq 0,\quad \left(\frac{\partial}{\partial x}\right)^\alpha W\in L^\infty(\mathbb{R}^N)\quad\left(|\alpha|\geq 2\right),\\
\label{pote-1}
W\in L^p(\mathbb{R}^N)+L^\infty(\mathbb{R}^N)\quad\left(p\geq 1\text{ and }p>\frac{N}{2}\right).
\end{align}

We define Hilbert spaces $\Sigma^k$ by
\begin{align*}
\Sigma^k:=\left\{u\in H^k(\mathbb{R}^N)\middle||x|^ku\in L^2(\mathbb{R}^N)\right\},\quad \|u\|_{\Sigma^k}^2:=\|u\|_{H^k}^2+\||\cdot|^k u\|_2^2.
\end{align*}

It is well known that \eqref{NLS} is locally well-posed in $\Sigma^1$ (see, e.g., \cite{CSSE,CHSEE}). This means that for any $u_0\in \Sigma^1$, there exists a unique maximal solution $u\in  C((T_*,T^*),\Sigma^1)\cap C^1((T_*,T^*),\Sigma^{-1})$. Moreover, the mass (i.e., $L^2$-norm) and energy $E$ of the solution are conserved by the flow, where 
\[
E(u):=\frac{1}{2}\left\|\nabla u\right\|_2^2-\frac{1}{2+\frac{4}{N}}\int_{\mathbb{R}^N}g(x)|u(x)|^{2+\frac{4}{N}}dx+\frac{1}{2}\int_{\mathbb{R}^N}W(x)|u(x)|^2dx.
\]
Furthermore, there is a blow-up alternative
\[
T^*<\infty\ \text{implies}\ \lim_{t\nearrow T^*}\left\|u(t)\right\|_{\Sigma^1}=\infty.
\]

Moreover, we consider the following condition instead of \eqref{pote-1}:
\begin{align}
\label{pote-2'}
W\in L^{p}(\mathbb{R}^N)+L^{\infty}(\mathbb{R}^N)\quad \left(p\geq 2\text{ and }p>\frac{N}{2}\right).
\end{align}
Under this condition, if $u_0\in \Sigma^2$, then the corresponding solution $u$ belongs to $u\in C((T_*,T^*),\Sigma^2)\cap C^1((T_*,T^*),L^2(\mathbb{R}^N))$.

In this paper, we investigate the conditions for the inhomogeneity and the potential related with the existence of minimal-mass blow-up solution.

\subsection{Critical problem}
Firstly, we describe the results regarding the mass-critical problem:
\begin{align}
\label{CNLS}
\tag{CNLS}
i\frac{\partial u}{\partial t}+\Delta u+|u|^{\frac{4}{N}}u=0,\quad (t,x)\in\mathbb{R}\times\mathbb{R}^N.
\end{align}

It is well known (\cite{BLGS,KGS,WGS}) that there exists a unique classical solution $Q$ for
\[
-\Delta Q+Q-\left|Q\right|^{\frac{4}{N}}Q=0,\quad Q\in H^1(\mathbb{R}^N),\quad Q>0,\quad Q\text{\ is\ radial},
\]
which is called the ground state. If $\|u\|_2=\|Q\|_2$ ($\|u\|_2<\|Q\|_2$, $\|u\|_2>\|Q\|_2$), we say that $u$ has the \textit{critical mass} (\textit{subcritical mass}, \textit{supercritical mass}, respectively).

We note that $E_{\text{crit}}(Q)=0$, where $E_{\text{crit}}$ is the energy with respect to \eqref{CNLS}. Moreover, the ground state $Q$ attains the best constant in the Gagliardo-Nirenberg inequality
\[
\left\|v\right\|_{2+\frac{4}{N}}^{2+\frac{4}{N}}\leq\left(1+\frac{2}{N}\right)\left(\frac{\left\|v\right\|_2}{\left\|Q\right\|_2}\right)^{\frac{4}{N}}\left\|\nabla v\right\|_2^2\quad\text{for }v\in H^1(\mathbb{R}^N).
\]
Therefore, for all $v\in H^1(\mathbb{R}^N)$,
\[
E_{\text{crit}}(v)\geq \frac{1}{2}\left\|\nabla v\right\|_2^2\left(1-\left(\frac{\left\|v\right\|_2}{\left\|Q\right\|_2}\right)^{\frac{4}{N}}\right)
\]
holds. This inequality and the mass and energy conservations imply that any subcritical-mass solution for \eqref{CNLS} is global and bounded in $H^1(\mathbb{R}^N)$.

Regarding the critical mass case, we apply the pseudo-conformal transformation
\[
u(t,x)\ \mapsto\ \frac{1}{\left|t\right|^\frac{N}{2}}u\left(-\frac{1}{t},\pm\frac{x}{t}\right)e^{i\frac{\left|x\right|^2}{4t}}
\]
to the solitary wave solution $u(t,x):=Q(x)e^{it}$. Then we obtain
\[
S(t,x):=\frac{1}{\left|t\right|^\frac{N}{2}}Q\left(\frac{x}{t}\right)e^{-\frac{i}{t}}e^{i\frac{\left|x\right|^2}{4t}},
\]
which is also a solution for \eqref{CNLS} and satisfies
\[
\left\|S(t)\right\|_2=\left\|Q\right\|_2,\quad \left\|\nabla S(t)\right\|_2\sim\frac{1}{\left|t\right|}\quad (t\nearrow 0).
\]
Namely, $S$ is a minimal-mass blow-up solution for \eqref{CNLS}. Moreover, $S$ is the only finite time blow-up solution for \eqref{CNLS} with critical mass, up to the symmetries of the flow (see \cite{MMMB}).

Regarding the supercritical mass case, there exists a solution $u$ for \eqref{CNLS} such that
\[
\left\|\nabla u(t)\right\|_2\sim\sqrt{\frac{\log\bigl|\log\left|T^*-t\right|\bigr|}{T^*-t}}\quad (t\nearrow T^*)
\]
(see \cite{MRUPB,MRUDB}).

\subsection{Previous results}
We describe previous results regarding  the following nonlinear Schr\"{o}dinger equation with a real-valued potential:
\begin{align}
\label{PNLS}
\tag{PNLS}
i\frac{\partial u}{\partial t}+\Delta u+|u|^{\frac{4}{N}}u-W(x)u=0,\quad (t,x)\in\mathbb{R}\times\mathbb{R}^N.
\end{align}

At first, \cite{C,CN} give results for unbounded potentials.

\begin{theorem}[Carles and Nakamura \cite{CN}]
\label{theorem:CN}
If $V(x)=E\cdot x$ for some $E\in\mathbb{R}^N$, then \eqref{PNLS} has a finite-time blow-up solution
\[
S(t,x):=\frac{1}{|t|^\frac{N}{2}}Q\left(\frac{x-t^2E}{t}\right)\exp\left({i\left(\frac{|x-t^2E|^2}{4t}-\frac{1}{t}+tE\cdot x-\frac{t^3}{3}|E|^2\right)}\right).
\]
In particular, $\|S\|_2=\|Q\|_2$.
\end{theorem}

\begin{theorem}[Carles \cite{C}]
\label{theorem:C}
If $W(x)=-\omega^2|x|^2$ for some $\omega\in\mathbb{R}^N$, then \eqref{PNLS} has a finite-time blow-up solution
\begin{align*}
S(t,x)&:=\left(\frac{2\omega}{\sinh \left(2\omega t\right)}\right)^\frac{N}{2}Q\left(\frac{2\omega x}{\sinh \left(2\omega t\right)}\right)\nonumber\\
&\hspace{20pt}\times\exp\left({i\left(\frac{\omega|x|^2}{2\sinh \left(2\omega t\right)\cosh \left(2\omega t\right)}-\frac{2\omega}{\tanh \left(2\omega t\right)}+\frac{\omega}{2}|x|^2\tanh \left(2\omega t\right)\right)}\right).
\end{align*}
In particular, $\|S\|_2=\|Q\|_2$.
\end{theorem}

Theorems \ref{theorem:CN} and \ref{theorem:C} construct blow-up solutions by applying the pseudo-conformal transformation to the ground states. Therefore, if \eqref{PNLS} can be reduced to \eqref{CNLS} (e.g., when $W$ is easy to handle algebraically), then \eqref{PNLS} may have a critical-mass blow-up solution with a blow-up rate of $t^{-1}$.

Merle \cite{MNE} and Rapha\"{e}l and Szeftel \cite{RSEU} consider
\begin{align}
\label{ICNLS}
\tag{ICNLS}
i\frac{\partial u}{\partial t}+\Delta u+g(x)|u|^{\frac{4}{N}}u=0,\quad (t,x)\in\mathbb{R}\times\mathbb{R}^N.
\end{align}

Firstly, \cite{MNE} showed non-existent results:

\begin{theorem}[\cite{MNE}]
\label{theorem:MNE}
Assume the following for $g$:
\begin{align}
\label{H.1}
&g_1\leq g\leq 1\quad\text{for some }g_1>0,\\
\label{H.2}
&g\in C^1(\mathbb{R}^N)\cap W^{1,\infty}(\mathbb{R}^N),\quad x\cdot\nabla g\in L^\infty(\mathbb{R}^N),\\
\label{H.3}
&g(x_0)=1\quad\text{for some }x_0\in\mathbb{R}^N,\\
&\exists \delta_0,R_0>0\ \forall |x|>R_0,\ g(x)\leq 1-\delta_0,\nonumber\\
&g^{-1}(\{1\})\text{ is finite},\nonumber\\
\label{H.4}
&\exists \rho_0>0,\alpha_0\in(0,1)\ \forall |x-x_0|\leq\rho_0,\ (x-x_0)\cdot\nabla g(x)\leq -|x-x_0|^{1+\alpha_0}.
\end{align}
Then there is no blow-up solutions with critical mass.
\end{theorem}

It is also shown that solutions for \eqref{ICNLS} with subcritical mass are globally in time if $g$ satisfies \eqref{H.1} and \eqref{H.2}. Moreover, it is additionally shown that if $k$ satisfies \eqref{H.3} and \eqref{H.4}, then there is a blow-up solution with supercritical mass less than $\|Q\|_2+\epsilon$ for some $\epsilon>0$. Thus, Theorem \ref{theorem:MNE} means that there is no minimal-mass blow-up solution at a finite time.

In contrast, \cite{RSEU} obtains results for existence:

\begin{theorem}[\cite{RSEU}]
\label{theorem:RSEU}
Assume $N=2$ and the following for $g$:
\begin{align*}
&g\in C^5(\mathbb{R}^2)\cap W^{1,\infty}(\mathbb{R}^2),\\
&g_1\leq g\leq 1\quad\text{for some }g_1>0\quad\text{and}\quad g(x_0)=1\quad\text{for some }x_0\in\mathbb{R}^N,\\
&\nabla^2g(x_0)<0.
\end{align*}
Then for any $E_0$ such that
\[
E_0>\frac{1}{8}\int_{\mathbb{R}^n}\nabla^2g(x_0)(y,y)Q(y)^4dy>0,
\]
there exist $t_0<0$ and a unique up to phase shift $u\in C([t_0,0),H^1(\mathbb{R}^2))$ that is solution for \eqref{ICNLS} with critical mass and energy $E_0$ and blows up at $t=0$.
\end{theorem}

The result differs from Theorems \ref{theorem:CN} and \ref{theorem:C} in that it does not use the classical method of pseudo-conformal transformation to construct the blow-up solution. Le Coz, Martel, and Rapha\"{e}l \cite{LMR}, based on the methodology of \cite{RSEU}, obtains the following results for
\begin{align}
\label{DPNLS}
\tag{DPNLS}
i\frac{\partial u}{\partial t}+\Delta u+|u|^{\frac{4}{N}}u\pm|u|^{p-1}u=0,\quad (t,x)\in\mathbb{R}\times\mathbb{R}^N.
\end{align}

Banica, Carles, and Duyckaerts \cite{BCD} presents the following result for
\begin{align}
\label{INLS}
\tag{INLS}
i\frac{\partial u}{\partial t}+\Delta u+g(x)|u|^{\frac{4}{N}}u-W(x)u=0,\quad (t,x)\in\mathbb{R}\times\mathbb{R}^N.
\end{align}

\begin{theorem}[\cite{BCD}]
\label{theorem:BCD}
Let $N=1$ or $2$, $W\in C^2(\mathbb{R}^N,\mathbb{R})$, and $g\in C^4(\mathbb{R}^N,\mathbb{R})$. Assume $\left(\frac{\partial}{\partial x}\right)^\beta W\in L^\infty(\mathbb{R}^N)\ (|\beta|\leq 2)$, $\left(\frac{\partial}{\partial x}\right)^\beta g\in L^\infty(\mathbb{R}^N)\ (|\beta|\leq 4)$, and
\[
g(0)=1,\quad \frac{\partial g}{\partial x_j}(0)=\frac{\partial^2 g}{\partial x_j\partial x_k}(0)=0\quad (1\leq j,k\leq N).
\]
Then there exist $T>0$ and a solution $u\in C((0,T),\Sigma^1)$ for \eqref{INLS} such that
\[
\left\|u(t)-\frac{1}{\lambda(t)^\frac{N}{2}}Q\left(\frac{x-x(t)}{\lambda(t)}\right)e^{i\frac{|x|^2}{4t}-i\theta\left(\frac{1}{t}\right)-itV(0)}\right\|_{\Sigma^1}\rightarrow 0\quad (t\searrow 0),
\]
where $\theta$ and $\lambda$ are continuous real-valued functions and $x$ is a continuous $\mathbb{R}^N$-valued function such that
\begin{align*}
&\theta(\tau)=\tau+o(\tau)\quad\text{as }\tau\rightarrow+\infty,\\
&\lambda(t)\sim t\text{ and }|x(t)|=o(t)\quad\text{as }t\searrow 0.
\end{align*}
\end{theorem}

\cite{NP} obtains the following result, which partially extends the result of \cite{BCD} using the method of \cite{LMR}.

\begin{theorem}[\cite{NP}]
\label{theorem:NP}
Let the potential $W$ satisfy 
\begin{align*}
&W\in C^{1,1}_{\text{loc}}(\mathbb{R}^N),\\
&\nabla W,\nabla^2 W\in L^q(\mathbb{R}^N)+L^{\infty}(\mathbb{R}^N)\quad \left(q\geq 2\ \text{and}\ q>N\right).
\end{align*}
Then there exist $t_0<0$ and a radial initial value $u_0\in \Sigma^1$ with $\|u_0\|_2=\|Q\|_2$ such that the corresponding solution $u$ for \eqref{PNLS} with $u(t_0)=u_0$ blows up at $t=0$. Moreover,
\[
\left\|u(t,x)-\frac{1}{\lambda(t)^\frac{N}{2}}Q\left(\frac{x+w(t)}{\lambda(t)}\right)e^{-i\frac{b(t)}{4}\frac{|x+w(t)|^2}{\lambda(t)^2}+i\gamma(t)}\right\|_{\Sigma^1}\rightarrow 0\quad (t\nearrow 0)
\]
holds for some $C^1$ functions $\lambda:(t_0,0)\rightarrow(0,\infty)$, $b,\gamma:(t_0,0)\rightarrow\mathbb{R}$, and $w:(t_0,0)\rightarrow\mathbb{R}^N$ such that
\[
\lambda(t)=|t|\left(1+o(1)\right),\quad b(t)=|t|\left(1+o(1)\right),\quad \gamma(t)\sim |t|^{-1},\quad |w(t)|=O(|t|^2)
\]
as $t\nearrow 0$.
\end{theorem}

\subsection{Main result}
In the main result, the following conditions are assumed.

The inhomogeneous function $g$ satisfies the following conditions:
\begin{align}
\label{G1}
\tag{G1}
&g\in W^{1,\infty}(\mathbb{R}^N),\quad x\cdot\nabla g\in L^\infty(\mathbb{R}^N),\\
\label{G2}
\tag{G2}
&|g(x)-1|\lesssim |x|^{2+r},\quad |\nabla g(x)|\lesssim |x|^{1+r}\quad(|x|\leq 1)
\end{align}
for some $r>0$.

We use the following notation
\[
X(f):=\left\{g:\text{measurable}\middle||g|\leq C f\text{ for some }C>0\right\}.
\]

The potential $W$ is the sum of potentials satisfying  \eqref{W1} or the following conditions:
\begin{align}
\label{W2}
\tag{W2}
&\left\{\begin{array}{l}
W\in L^{p_1}(\mathbb{R}^N)+L^{\infty}(\mathbb{R}^N)\quad \left(p_1\geq 2\text{ and }p_1>\frac{N}{2}\right),\\
\nabla W\in L^{p_2}(\mathbb{R}^N)+X(1+|\cdot|)\quad\left(p_2\geq 2\text{ and }p_2>N\right),
\end{array}\right.
\end{align}
and furthermore satisfies one of the followings:
\begin{align}
\label{W2-1}
\tag{W2-1}
W&\text{ is locally Lipschitz continuous},\\
\label{W2-2}
\tag{W2-2}
W&\in X(|\cdot|^{r'}e^{C|\cdot|})\quad\text{for some }C,r'>0.
\end{align}
Namely, $W$ is the sum of potentials satisfying \eqref{W1}, \eqref{W2} and \eqref{W2-1}, or \eqref{W2} and \eqref{W2-2}.

\begin{theorem}[Existence of a critical-mass blow-up solution]
\label{theorem:EMBS}
For any energy level $E_0>0$, there exist $t_0<0$ and a radial initial value $u_0\in \Sigma^1$ with $\|u_0\|_2=\|Q\|_2$ and $E(u_0)=E_0$ such that the corresponding solution $u$ for \eqref{NLS} with $u(t_0)=u_0$ blows up at $t=0$. Moreover,
\[
\left\|u(t,x)-\frac{1}{\lambda(t)^\frac{N}{2}}Q\left(\frac{x+w(t)}{\lambda(t)}\right)e^{-i\frac{b(t)}{4}\frac{|x+w(t)|^2}{\lambda(t)^2}+i\gamma(t)}\right\|_{\Sigma^1}\rightarrow 0\quad (t\nearrow 0)
\]
holds for some $C^1$ functions $\lambda:(t_0,0)\rightarrow(0,\infty)$, $b,\gamma:(t_0,0)\rightarrow\mathbb{R}$, and $w:(t_0,0)\rightarrow\mathbb{R}^N$ such that
\[
\lambda(t)=\sqrt{\frac{8E_0}{\|yQ\|_2^2}}|t|\left(1+o(1)\right),\quad b(t)=\frac{8E_0}{\|yQ\|_2^2}|t|\left(1+o(1)\right),\quad \gamma(t)\sim |t|^{-1},\quad |w(t)|=o(|t|)
\]
as $t\nearrow 0$.
\end{theorem}

\begin{remark}
In contrast, if $g\leq 1$ and $W$ satisfies \eqref{W1} or \eqref{pote-1}, then any subcritical-mass solution for \eqref{NLS} exists globally in time and is bounded in $H^1$. This can be proved easily by the Gagliardo-Nirenberg inequality and the Sobolev embedding theorem. Therefore, the solution in Theorem \ref{theorem:EMBS} is a minimal-mass blow-up solution if $g\leq 1$.
\end{remark}

\subsection{Comments regarding the main result}
Theorem \ref{theorem:EMBS} is a generalisation of Theorems \ref{theorem:BCD} and \ref{theorem:NP}.

$W^{1,\infty}_{\loc}$ can be regarded as $C^{0,1}_{\loc}$ by identifying the difference on the null set. Thus, $|g(x)-1|\lesssim |x|^{2+r}$ in \eqref{G2} may be replaced by $g(x)=1$.

From the assumption \eqref{H.4} in Theorem \ref{theorem:MNE}, which is the nonexistence result, we obtain
\[
|x|^{\alpha_0}\leq |\nabla g(x)|\quad\text{for }|x|\leq\rho_0
\]
for some $\alpha_0\in(0,1)$, where we assume $g(0)=1$. In contrast, Theorem \ref{theorem:EMBS}, which is the existence result, assumes
\[
|\nabla g(x)|\lesssim |x|^{1+r}\quad\text{for }|x|\leq1
\]
for some $r>0$. Therefore, the threshold for the existence and non-existence of blow-up solutions with critical mass can be said to be $\alpha_0=1$ (i.e., $r=0$). The result in the case of the threshold has been obtained in part by Theorem \ref{theorem:RSEU}.

From the point of view of differentiability, it seems that neither $g$ nor $W$ need to be smooth over the whole $\mathbb{R}^N$, since blow-up is crucial for behaviour in the neighbourhood of the blow-up point. On the other hand, first-order differentiations is necessary for the technicality of the proof. Thus, the assumption that $g$ and $W$ are first-order weakly differentiable would be quite close to the limit.

Compared to Theorem \ref{theorem:NP}, Theorem \ref{theorem:EMBS} requires less order of differentiation for the potential $W$. In \cite{NP}, the bootstrap of $\lambda$ and $b$ is done by differentiating and then integrating, thus the condition $\frac{\partial b}{\partial s}+b^2=o(s^{-3})$ is required. Thus, \cite{NP} has required $C^{1,1}_{\loc}$ for $W$. However, in this paper, the condition is removed by using the property of energy. Consequently, we reduce the order of differentiation.

From the point of view of integrability, it would be possible to replace \eqref{G1} and \eqref{W2} with weaker conditions. In fact, a scrutiny of proofs of Proposition \ref{Psiesti}, Lemma \ref{Lambda}, etc. shows that some of them can be substituted by other integrable conditions in their proofs. However, it would be complex to attempt to describe them exhaustively.

\section{Notation and preliminaries}
\label{sec:Preliminaries}
We define
\begin{align*}
&(u,v)_2:=\re\int_{\mathbb{R}^N}u(x)\overline{v}(x)dx,\quad \left\|u\right\|_p:=\left(\int_{\mathbb{R}^N}|u(x)|^pdx\right)^\frac{1}{p},\\
&f(z):=|z|^\frac{4}{N}z,\quad  F(z):=\frac{1}{2+\frac{4}{N}}|z|^{2+\frac{4}{N}}\quad \text{for $z\in\mathbb{C}$}.
\end{align*}
By identifying $\mathbb{C}$ with $\mathbb{R}^2$, we denote the differentials of $f$ and $F$ by $df$ and $dF$, respectively. We define
\[
\Lambda:=\frac{N}{2}+x\cdot\nabla,\quad L_+:=-\Delta+1-\left(1+\frac{4}{N}\right)Q^\frac{4}{N},\quad L_-:=-\Delta+1-Q^\frac{4}{N}.
\]
Namely, $\Lambda$ is the generator of $L^2$-scaling, and $L_+$ and $L_-$ come from the linearised Schr\"{o}dinger operator to close $Q$. Then
\[
L_-Q=0,\quad L_+\Lambda Q=-2Q,\quad L_-|x|^2Q=-4\Lambda Q,\quad L_+\rho=|x|^2 Q,\quad L_-xQ=-\nabla Q,\quad L_+\nabla Q=0
\]
hold, where $\rho\in\mathcal{S}(\mathbb{R}^N)$ is the unique radial solution for $L_+\rho=|x|^2 Q$. Note that there exist $C_\alpha,\kappa_\alpha>0$ such that
\[
\left|\left(\frac{\partial}{\partial x}\right)^\alpha Q(x)\right|\leq C_\alpha Q(x),\quad \left|\left(\frac{\partial}{\partial x}\right)^\alpha \rho(x)\right|\leq C_\alpha(1+|x|)^{\kappa_\alpha} Q(x).
\]
for any multi-index $\alpha$. Furthermore, there exists $\mu>0$ such that for any $u\in H^1(\mathbb{R}^N)$,
\begin{align}
\label{Lcoer}
&\left\langle L_+\re u,\re u\right\rangle+\left\langle L_-\im u,\im u\right\rangle\nonumber\\
\geq&\ \mu\left\|u\right\|_{H^1}^2-\frac{1}{\mu}\left({(\re u,Q)_2}^2+\left|(\re u,xQ)_2\right|^2+{(\re u,|x|^2 Q)_2}^2+{(\im u,\rho)_2}^2\right)
\end{align}
holds (see, e.g., \cite{MRO,MRUPB,RSEU,WL}). Finally, we use the notation $\lesssim$ and $\gtrsim$ when the inequalities hold up to a positive constant. We also use the notation $\approx$ when $\lesssim$ and $\gtrsim$ hold.

We estimate the error terms $\Psi$ that is defined by
\[
\Psi(y;\lambda,b):=\lambda^2W(\lambda y-w)Q(y).
\]
Moreover, we define $\kappa$ by
\[
\kappa:=\min\left\{1,2-\frac{N}{p_1},1-\frac{N}{p_2},r,r'\right\}\in(0,1].
\]
Without loss of generality, we may assume that $W(0)=0$. In particular, if $W$ satisfies \eqref{W1},
\[
W\in X(|\cdot|+|\cdot|^2),\quad \nabla W\in X(1+|\cdot|)
\]
holds.

\begin{proposition}[Estimate of $\Psi$]
\label{Psiesti}
There exists a sufficiently small constant $\epsilon'>0$ such that 
\[
\left\|e^{\epsilon'|y|}\Psi\right\|_{H^1}\lesssim\lambda^{1+\kappa}(\lambda+|w|)
\]
for $0<\lambda\ll 1$ and $w\in\mathbb{R}^N$ such that $|w|\leq 1$.
\end{proposition}

\begin{proof}
From the assumptions for $W$, we can write $W=W_1+W_2$ and $W_2=W_{21}+W_{22}$ using $W_1$, $W_2$, $W_{21}$, and $W_{22}$ satisfying \eqref{W1}, \eqref{W2}, \eqref{W2-1}, and \eqref{W2-2}, respectively.

Firstly, since $W_1\in X(|\cdot|+|\cdot|^2)$ and $\nabla W_1\in X(1+|\cdot|)$, we obtain
\begin{align*}
\|e^{\epsilon'|\cdot|}\lambda^2W_1(\lambda\cdot-w)Q\|_2&\lesssim\|e^{\epsilon'|\cdot|}\lambda^2(\lambda|\cdot|+\lambda^2|\cdot|^2+|w|)Q\|_2\lesssim \lambda^2(\lambda+|w|),\\
\|e^{\epsilon'|\cdot|}\lambda^3\nabla W_1(\lambda\cdot-w)Q\|_2&\lesssim\|e^{\epsilon'|\cdot|}\lambda^3(1+\lambda|\cdot|+|w|)Q\|_2\lesssim \lambda^3.
\end{align*}

Secondly, since
\[
W_{21}(\lambda y-w)=\int_0^1(\lambda y-w)\cdot\nabla W_{21}(\tau(\lambda y-w))d\tau,
\]
we obtain
\begin{align*}
\|e^{\epsilon'|\cdot|}\lambda^2W_{21}(\lambda\cdot-w)Q\|_2&\lesssim \lambda^{2-\frac{N}{p_2}}(\lambda+|w|)+\lambda^2(\lambda+|w|),\\
\|e^{\epsilon'|\cdot|}\lambda^3\nabla W_{21}(\lambda\cdot-w)Q\|_2&\lesssim \lambda^{3-\frac{N}{p_2}}+\lambda^3.
\end{align*}

Finally,
\begin{align*}
\|e^{\epsilon'|\cdot|}\lambda^2W_{22}(\lambda\cdot-w)Q\|_2&\lesssim \|e^{\epsilon'|\cdot|}\lambda^2(\lambda^r|\cdot|^r+|w|^r)e^{C(\lambda|\cdot|+|w|)}Q\|_2\lesssim\lambda^2(\lambda^r+|w|^r)\lesssim \lambda^{1+r}(\lambda+|w|),\\
\|e^{\epsilon'|\cdot|}\lambda^3\nabla W_{21}(\lambda\cdot-w)Q\|_2&\lesssim \lambda^{3-\frac{N}{p_2}}.
\end{align*}
\end{proof}

\begin{remark*}
This estimate holds true even if $Q$ is replaced by $|\cdot|^2$, $\rho$, etc.
\end{remark*}

Furthermore, direct calculations yield the following properties:

\begin{proposition}
\label{profileprop}
Let
\[
Q_{\lambda,b,w,\gamma}(x):=\frac{1}{\lambda^\frac{N}{2}}Q\left(\frac{x+w}{\lambda}\right)e^{-i\frac{b}{4}\frac{|x+w|^2}{\lambda^2}+i\gamma}.
\]
Then
\[
\left|8E(Q_{\lambda,b,w,\gamma})-\frac{b^2}{\lambda^2}\|yQ\|_2^2\right|\lesssim\frac{\lambda^{2+\kappa}+|w|^{2+\kappa}}{\lambda^2}.
\]
holds for $0<\lambda\ll1$ and $w\in\mathbb{R}^N$ such that $|w|\leq 1$, where $y=\frac{x+w}{\lambda}$.

Moreover, if $s\mapsto(\lambda(s),b(s),w(s))$ is $C^1$-function,
\[
\left|\frac{d}{ds}E(Q_{\lambda,b,w,\gamma})\right|\lesssim\frac{1}{\lambda^2}\left(\left(\lambda^{1+\kappa}+|b|+|w|^{1+\kappa}\right)\left(\left|\frac{1}{\lambda}\frac{\partial\lambda}{\partial s}+b\right|+\left|\frac{\partial b}{\partial s}+b^2\right|+\left|\frac{\partial w}{\partial s}\right|\right)+|b|(\lambda^{2+\kappa}+|w|^{2+\kappa})\right)
\]
holds.
\end{proposition}

At the end of this section, we state the following standard result. For the proof, see \cite{MRUPB}.

\begin{lemma}[Decomposition]
\label{decomposition}
There exists $\overline{C}>0$ such that the following statement holds. Let $I$ be an interval and $\delta>0$ be sufficiently small. We assume that $u\in C(I,H^1(\mathbb{R}^N))\cap C^1(I,\Sigma^{-1})$ satisfies
\[
\forall\ t\in I,\ \left\|\lambda(t)^{\frac{N}{2}}u\left(t,\lambda(t)y-w(t)\right)e^{i\gamma(t)}-Q\right\|_{H^1}< \delta
\]
for some functions $\lambda:I\rightarrow(0,\infty)$, $\gamma:I\rightarrow\mathbb{R}$, and $w:I\rightarrow\mathbb{R}^N$. Then there exist unique functions $\tilde{\lambda}:I\rightarrow(0,\infty)$, $\tilde{b}:I\rightarrow\mathbb{R}$, $\tilde{\gamma}:I\rightarrow\mathbb{R}\slash 2\pi\mathbb{Z}$, and $\tilde{w}:I\rightarrow\mathbb{R}^N$ such that 
\begin{align}
\label{mod}
&u(t,x)=\frac{1}{\tilde{\lambda}(t)^{\frac{N}{2}}}\left(Q+\tilde{\varepsilon}\right)\left(t,\frac{x+\tilde{w}(t)}{\tilde{\lambda}(t)}\right)e^{-i\frac{\tilde{b}(t)}{4}\frac{|x+\tilde{w}(t)|^2}{\tilde{\lambda}(t)^2}+i\tilde{\gamma}(t)},\\
&\left|\frac{\tilde{\lambda}(t)}{\lambda(t)}-1\right|+\left|\tilde{b}(t)\right|+\left|\tilde{\gamma}(t)-\gamma(t)\right|_{\mathbb{R}\slash 2\pi\mathbb{Z}}+\left|\frac{\tilde{w}(t)-w(t)}{\tilde{\lambda}(t)}\right|<\overline{C}\nonumber
\end{align}
hold, where $|\cdot|_{\mathbb{R}\slash 2\pi\mathbb{Z}}$ is defined by
\[
|c|_{\mathbb{R}\slash 2\pi\mathbb{Z}}:=\inf_{m\in\mathbb{Z}}|c+2\pi m|,
\]
and that $\tilde{\varepsilon}$ satisfies the orthogonal conditions
\begin{align}
\label{orthocondi}
\left(\tilde{\varepsilon},i\Lambda Q\right)_2=\left(\tilde{\varepsilon},|y|^2Q\right)_2=\left(\tilde{\varepsilon},i\rho\right)_2=0,\quad \left(\tilde{\varepsilon},yQ\right)_2=0
\end{align}
on $I$. In particular, $\tilde{\lambda}$, $\tilde{b}$, $\tilde{\gamma}$, and $\tilde{w}$ are $C^1$ functions and independent of $\lambda$, $\gamma$, and $w$.
\end{lemma}

\section{Uniformity estimates for modulation terms}
\label{sec:uniesti}
From this section to Section \ref{sec:bootstrap}, we prepare lemmas for the proof of Theorem \ref{theorem:EMBS}.

For $s_1>0$, let $\lambda_1,b_1>0$ be defined by
\[
\lambda_1:=\sqrt{\frac{\|yQ\|_2^2}{8E_0}}{s_1}^{-1},\quad E(Q_{\lambda_1,b_1,0,0})=E_0.
\]

Let $u(t)$ be the solution for \eqref{NLS} with an initial value
\begin{align}
\label{initial}
u(t_1,x):=\frac{1}{{\lambda_1}^\frac{N}{2}}Q\left(\frac{x}{\lambda_1}\right)e^{-i\frac{b_1}{4}\frac{|x|^2}{{\lambda_1}^2}}.
\end{align}
Note that $u\in C((T_*,T^*),\Sigma^2(\mathbb{R}^N))$ and $|x|\nabla u\in C((T_*,T^*),L^2(\mathbb{R}^N))$. Moreover,
\[
\im\int_{\mathbb{R}^N}u(t_1,x)\nabla\overline{u}(t_1,x)dx=0
\]
holds.

Since $u$ satisfies the assumption in Lemma \ref{decomposition} in a neighbourhood of $t_1$, there exist decomposition parameters $\tilde{\lambda}_{t_1}$, $\tilde{b}_{t_1}$, $\tilde{\gamma}_{t_1}$, $\tilde{w}_{t_1}$, and $\tilde{\varepsilon}_{t_1}$ such that \eqref{mod} and \eqref{orthocondi} hold in the neighbourhood. We define the rescaled time $s_{t_1}$ by
\[
s_{t_1}(t):=s_1-\int_t^{t_1}\frac{1}{\tilde{\lambda}_{t_1}(\tau)^2}d\tau.
\]
Moreover, we define
\begin{align*}
&t_{t_1}:=-\frac{\|yQ\|_2^2}{8E_0}{s_{t_1}}^{-1}, \quad\lambda_{t_1}(s):=\tilde{\lambda}_{t_1}(t_{t_1}(s)),\quad b_{t_1}(s):=\tilde{b}_{t_1}(t_{t_1}(s)),\\
&\gamma_{t_1}(s):=\tilde{\gamma}_{t_1}(t_{t_1}(s)),\quad w_{t_1}(s):=\tilde{w}_{t_1}(t_{t_1}(s)),\quad \varepsilon_{t_1}(s,y):=\tilde{\varepsilon}_{t_1}(t_{t_1}(s),y).
\end{align*}
For the sake of clarity in notation, we often omit the subscript $t_1$. Furthermore, let $I_{t_1}$ be the maximal interval of the existence of the decomposition such that \eqref{mod} and \eqref{orthocondi} hold and we define 
\[
J_{s_1}:=s_{t_1}\left(I_{t_1}\right).
\]
Additionally, let $s_0$ be sufficiently large, $s_1>s_0$, and
\[
s':=\max\left\{s_0,\inf J_{s_1}\right\}.
\]
In particular,
\begin{align}
\label{epsieq}
\Psi&=i\frac{\partial \varepsilon}{\partial s}+\Delta \varepsilon-\varepsilon+g(\lambda y-w)f\left(Q+\varepsilon\right)-f\left(Q\right)-\lambda^2 W(\lambda y-w)\varepsilon\\
&\hspace{30pt}-i\left(\frac{1}{\lambda}\frac{\partial \lambda}{\partial s}+b\right)\Lambda (Q+\varepsilon)+\left(1-\frac{\partial \gamma}{\partial s}\right)(Q+\varepsilon)+\left(\frac{\partial b}{\partial s}+b^2\right)\frac{|y|^2}{4}(Q+\varepsilon)\nonumber\\
&\hspace{50pt}-\left(\frac{1}{\lambda}\frac{\partial \lambda}{\partial s}+b\right)b\frac{|y|^2}{2}(Q+\varepsilon)+i\frac{1}{\lambda}\frac{\partial w}{\partial s}\cdot\nabla(Q+\varepsilon)+\frac{1}{2}\frac{b}{\lambda}\frac{\partial w}{\partial s}\cdot y(Q+\varepsilon)\nonumber
\end{align}
holds in $J_{s_{t_1}}$.

Let $L$ be defined by
\[
L:=1+\frac{\kappa}{2}
\]
Moreover, we define $s_*$ by
\[
s_*:=\inf\left\{\sigma\in(s',s_1]\ \middle|\ \text{\eqref{bootstrap} holds on }[\sigma,s_1]\right\},
\]
where
\begin{align}
\label{bootstrap}
\left\|\varepsilon(s)\right\|_{H^1}^2+b(s)^2\|y\varepsilon(s)\|_2^2<s^{-2L},\quad\left|\frac{\lambda(s)}{\lambda_{\app}(s)}-1\right|+\left|\frac{b(s)}{b_{\app}(s)}-1\right|<s^{-\frac{\kappa}{2}},\quad |w(s)|<s^{-(1+\frac{\kappa}{2})}.
\end{align}

Finally, we define
\[
\text{Mod}(s):=\left(\frac{1}{\lambda}\frac{\partial \lambda}{\partial s}+b,\frac{\partial b}{\partial s}+b^2,1-\frac{\partial \gamma}{\partial s},\frac{\partial w}{\partial s}\right).
\]
The goal of this section is to estimate of $\text{Mod}(s)$.

\begin{lemma}
For $s\in(s_*,s_1]$,
\begin{align}
\label{nablaortho}
\left|(\im\varepsilon(s),\nabla Q)_2\right|\lesssim s^{-(2L-1)}.
\end{align}
\end{lemma}

\begin{proof}
According to a direct calculation, we have
\[
\frac{d}{dt}\im\int_{\mathbb{R}^N}u(t,x)\nabla\overline{u}(t,x)dx=\int_{\mathbb{R}^N}\left(-\frac{1}{1+\frac{2}{N}}\nabla g(x)|u(t,x)|^{2+\frac{4}{N}}+\frac{1}{2}\nabla W(x)|u(t,x)|^2\right)dx.
\]
Since
\begin{align*}
\left|\lambda^2\int_{\mathbb{R}^N}\nabla g(x)|u(t(s),x)|^{2+\frac{4}{N}}dx\right|&=\left|\int_{\mathbb{R}^N}\nabla g(\lambda y-w)|Q(y)+\varepsilon(s,y)|^{2+\frac{4}{N}}dy\right|\lesssim \lambda^{1+\kappa}+|w|^{1+\kappa}+\|\varepsilon\|_{2+\frac{4}{N}}^{2+\frac{4}{N}},\\
\lambda^2\int_{\mathbb{R}^N}\nabla W(x)|u(t(s),x)|^2dx&=\lambda^2\int_{\mathbb{R}^N}\nabla W(\lambda y-w)|Q(y)+\varepsilon(s,y)|^2dy,\\
\left|\lambda^2\int_{\mathbb{R}^N}\nabla W(\lambda y-w)Q(y)^2dy\right|&\lesssim\frac{1}{\lambda}\|\Psi\|_{H^1},\\
\left|\lambda^2\int_{\mathbb{R}^N}\nabla W_1(\lambda y-w)|\varepsilon(s,y)|^2dy\right|&\lesssim\lambda^2\|\varepsilon\|_2(\|\varepsilon\|_2+b\|y\varepsilon\|_2),\\
\left|\lambda^2\int_{\mathbb{R}^N}\nabla W_2(\lambda y-w)|\varepsilon(s,y)|^2dy\right|&\lesssim\lambda^{2-\frac{N}{p_2}}\|\varepsilon\|_{H^1}^2+\lambda^2\|\varepsilon\|_2(\|\varepsilon\|_2+b\|y\varepsilon\|_2),
\end{align*}
we obtain
\[
\left|\frac{d}{ds}\im\int_{\mathbb{R}^N}u(t(s),x)\nabla\overline{u}(t(s),x)dx\right|\lesssim \lambda^2\left|\frac{d}{dt}\im\int_{\mathbb{R}^N}u(t,x)\nabla\overline{u}(t,x)dx\right|\lesssim s^{-(1+\kappa)}.
\]
Therefore, we obtain
\[
\left|\im\int_{\mathbb{R}^N}u(t(s),x)\nabla\overline{u}(t(s),x)dx\right|\lesssim s^{-\kappa}\lesssim s^{-2(L-1)}.
\]

The rest is shown in the same way as in \cite[Lemma 3.2]{NP}.
\end{proof}

\begin{lemma}[Estimation of modulation terms]
For $s\in(s_*,s_1]$,
\begin{align}
\label{ortho}
(\varepsilon(s),Q)_2&=-\frac{1}{2}\left\|\varepsilon(s)\right\|_2^2,\\
\label{modesti}
\left|\text{Mod}(s)\right|&\lesssim s^{-2L}
\end{align}
holds.
\end{lemma}

\begin{proof}
According to the mass conservation, we have
\[
\left(\varepsilon,Q\right)_2=\frac{1}{2}\left(\left\|u\right\|_2^2-\left\|Q\right\|_2^2-\left\|\varepsilon\right\|_2^2\right)=-\frac{1}{2}\left\|\varepsilon\right\|_2^2.
\]
meaning $(\ref{ortho})$ holds.

For $v=\Lambda Q$, $i|y|^2Q$, $\rho$, or $y_jQ$, the following estimates hold:
\begin{align*}
|(g(\lambda y-w)-1)f(Q+\varepsilon)||v|&\lesssim(\lambda^{2+\kappa}+|w|^{2+\kappa})(Q+|\varepsilon|)|v|^\frac{1}{2}\\
|f\left(Q+\varepsilon\right)-f\left(Q\right)-df(Q)(\varepsilon)||v|\lesssim |\varepsilon|^2,\\
|(\lambda^2 W(\lambda y-w)\varepsilon,v)_2|&\lesssim \lambda^{1+\kappa}\left(\lambda+|w|\right)\|\varepsilon\|_2.
\end{align*}
Therefore, according to orthogonal conditions \eqref{orthocondi}, Equation \eqref{epsieq}, Proposition \ref{Psiesti}, and \eqref{nablaortho}, we see that
\[
\left|\text{Mod}(s)\right|\lesssim s^{-2L}+\epsilon\left|\text{Mod}(s)\right|.
\]
For detail of the proof of the inequality, see \cite[Lemma 4.1]{LMR}. Consequently, we obtain \eqref{modesti}.
\end{proof}

\section{Modified energy function}
\label{sec:MEF}
In this section, we proceed with a modified version of the technique presented in Le Coz, Martel, and Rapha\"{e}l \cite{LMR} and Rapha\"{e}l and Szeftel \cite{RSEU}. Let $m$ and $\epsilon_j$ be defined by
\begin{align}
\label{epsidef}
m:=2+\frac{\kappa}{2},\quad \epsilon_1:=\frac{\kappa m\mu}{32},\quad \epsilon_3:=\min\left\{\frac{\mu}{24},\frac{\kappa^2\mu}{24\times 64}\right\},\quad \epsilon_4:=\min\left\{\frac{m\mu}{24},\frac{\kappa^2m\mu}{24\times 64}\right\},\quad \epsilon_5:=\frac{\kappa}{8}
\end{align}
where $\mu$ is from the coercivity \eqref{Lcoer} of $L_+$ and $L_-$. Moreover, we define
\begin{align*}
H(s,\varepsilon)&:=\frac{1}{2}\left\|\varepsilon\right\|_{H^1}^2+\frac{\epsilon_1b^2}{2}\left\|y\varepsilon\right\|_2^2-\int_{\mathbb{R}^N}g(\lambda y-w)\left(F(Q(y)+\varepsilon(y))-F(Q(y))-dF(Q(y))(\varepsilon(y))\right)dy\\
&\hspace{30pt}+\frac{1}{2}\lambda^2\int_{\mathbb{R}^N}W(y)|\varepsilon(y)|^2dy,\\
S(s,\varepsilon)&:=\frac{1}{\lambda^m}H(s,\varepsilon).
\end{align*}

\begin{lemma}[Coercivity of $H$]
\label{Hcoer}
For $s\in(s_*,s_1]$, 
\[
H(s,\varepsilon)\geq \frac{\mu}{2}\|\varepsilon\|_{H^1}^2+\frac{\epsilon_1}{2}b^2\left\|y\varepsilon\right\|_2^2-\epsilon_3\left(\|\varepsilon\|_{H^1}^2+b^2\left\|y\varepsilon\right\|_2^2\right)
\]
holds.
\end{lemma}

\begin{proof}
Firstly, we have
\begin{align*}
\left|\lambda^2\int_{\mathbb{R}^N}W_1(\lambda y-w)|\varepsilon|^2dy\right|&\lesssim \lambda^2\left(\|\varepsilon\|_{H^1}^2+b^2\|y\varepsilon\|_2^2\right),\\
\left|\lambda^2\int_{\mathbb{R}^N}W_2(\lambda y-w)|\varepsilon|^2dy\right|&\lesssim \lambda^{2-\frac{N}{p_1}}\|\varepsilon\|_{H^1}^2+\lambda^2\|\varepsilon\|_2^2.
\end{align*}

Secondly,
\[
\left|\int_{\mathbb{R}^N}g(\lambda y-w)\left(F(Q+\varepsilon)-F(Q)-dF(Q)(\varepsilon)-\frac{1}{2}d^2F(Q)(\varepsilon,\varepsilon)\right)dy\right|\lesssim \|\varepsilon\|_{H^1}^3+\|\varepsilon\|_{H^1}^{2+\frac{4}{N}}.
\]

Thirdly,
\[
\left|\int_{\mathbb{R}^N}(g(\lambda y-w)-1)d^2F(Q)(\varepsilon,\varepsilon)dy\right|\lesssim s^{-(2+\kappa)}\|\varepsilon\|_{H^1}^2.
\]

Finally, from \eqref{Lcoer}, \eqref{orthocondi}, and \eqref{ortho} since
\[
\left\|\varepsilon\right\|_{H^1}^2-\int_{\mathbb{R}^N}d^2F(Q)(\varepsilon,\varepsilon)dy=\left(L_+\re\varepsilon,\re\varepsilon\right)_2+\left( L_-\im\varepsilon,\im\varepsilon\right)_2,
\]
we have
\[
H(s,\varepsilon)\geq\frac{\mu}{2}\|\varepsilon\|_{H^1}^2+\frac{\epsilon_1b^2}{2}\left\|y\varepsilon\right\|_2^2-\epsilon_3\left(\|\varepsilon\|_2^2+b^2\|y\varepsilon\|_2^2\right).
\]
Consequently, we obtain Lemma \ref{Hcoer} if $s_0$ is sufficiently large.
\end{proof}

\begin{corollary}[Estimation of $S$]
\label{Sesti}
For $s\in(s_*,s_1]$, 
\[
\frac{1}{\lambda^m}\left(\|\varepsilon\|_{H^1}^2+b^2\left\|y\varepsilon\right\|_2^2\right)\lesssim S(s,\varepsilon)\lesssim\frac{1}{\lambda^m}\left(\|\varepsilon\|_{H^1}^2+b^2\left\|y\varepsilon\right\|_2^2\right)
\]
holds.
\end{corollary}

\begin{lemma}
\label{Lambda}
For all $s\in(s_*,s_1]$, 
\begin{align}
\label{flambdaesti}
\left|\left(g(\lambda y-w)\left(f(Q+\varepsilon)-f(Q)\right),\Lambda \varepsilon\right)_2\right|&\lesssim \|\varepsilon\|_{H^1}^2,\\
\label{fdefesti}
\left|\left(g(\lambda y-w)\left(f(Q+\varepsilon)-f(Q)\right),\nabla \varepsilon\right)_2\right|&\lesssim \|\varepsilon\|_{H^1}^2,\\
\label{Vlambdaesti}
\left|\lambda^2\left(W(\lambda y-w)\varepsilon,\Lambda\varepsilon\right)_2\right|&\lesssim s^{-1}\left(\|\varepsilon\|_{H^1}^2+b^2\|y\varepsilon\|_2^2\right),\\
\label{Vdefesti}
\left|\lambda^2\left(W(\lambda y-w)\varepsilon,\nabla\varepsilon\right)_2\right|&\lesssim s^{-1}\left(\|\varepsilon\|_{H^1}^2+b^2\|y\varepsilon\|_2^2\right).
\end{align}
\end{lemma}

\begin{proof}
For \eqref{Vlambdaesti} and \eqref{Vdefesti}, see \cite{NP}.

Firstly,
\begin{align*}
&\nabla \left(g(\lambda y-w)\left(F(Q+\varepsilon)-F(Q)-dF(Q)(\varepsilon)\right)\right)\\
=&\lambda(\nabla g)(\lambda y-w)\left(F(Q+\varepsilon)-F(Q)-dF(Q)(\varepsilon)\right)+g(\lambda y-w)\re\left(f(Q+\varepsilon)-f(Q)-df(Q)(\varepsilon)\right)\nabla Q\\
&\hspace{20pt}+g(\lambda y-w)\re\left(\left(f(Q+\varepsilon)-f(Q)\right)\nabla\overline{\varepsilon}\right)
\end{align*}
Therefore, we obtain
\begin{align*}
&g(\lambda y-w)\re\left(\left(f(Q+\varepsilon)-f(Q)\right)\Lambda\overline{\varepsilon}\right)\\
=&\frac{N}{2}g(\lambda y-w)\re\left(\left(f(Q+\varepsilon)-f(Q)\right)\overline{\varepsilon}\right)+y\cdot\nabla\left(g(\lambda y-w)\left(F(Q+\varepsilon)-F(Q)-dF(Q)(\varepsilon)\right)\right)\\
&\hspace{20pt}-w\cdot(\nabla g)(\lambda y-w)\left(F(Q+\varepsilon)-F(Q)-dF(Q)(\varepsilon)\right)-(\lambda y-w)\cdot(\nabla g)(\lambda y-w)\left(F(Q+\varepsilon)-F(Q)-dF(Q)(\varepsilon)\right)\\
&\hspace{40pt}+g(\lambda y-w)\re\left(f(Q+\varepsilon)-f(Q)-df(Q)(\varepsilon)\right)y\cdot\nabla Q.
\end{align*}
Therefore, we obtain
\begin{align*}
\left|\left(g(\lambda y-w)\left(f(Q+\varepsilon)-f(Q)\right),\Lambda \varepsilon\right)_2\right|&\lesssim \|\varepsilon\|_{H^1}^2\\
\end{align*}
so that \eqref{flambdaesti} holds. \eqref{fdefesti} is also shown by similar calculations.
\end{proof}

\begin{lemma}[Derivative of $H$ in time]
\label{Hdef}
For all $s\in(s_*,s_1]$, 
\[
\frac{d}{ds}H(s,\varepsilon(s))\geq -b\left(\left(\frac{\epsilon_1}{\epsilon_5}+\epsilon_4\right)\|\varepsilon\|_{H^1}^2+\left(1+\frac{\epsilon_4}{\epsilon_1}+\epsilon_5\right)\epsilon_1b^2\left\|y\varepsilon\right\|_2^2+Cs^{-(2+\kappa)}\right).
\]
\end{lemma}

\begin{proof}
Outline the proofs. See \cite{LMR} for details.

Firstly, we have
\[
\frac{d}{ds}H(s,\varepsilon(s))=\frac{\partial H}{\partial s}(s,\varepsilon(s))+\left\langle i\frac{\partial H}{\partial \varepsilon},i\frac{\partial \varepsilon}{\partial s}\right\rangle.
\]
Secondly, we have
\begin{align*}
\frac{\partial H}{\partial \varepsilon}&=-\Delta \varepsilon+\varepsilon+\epsilon_1b^2|y|^2\varepsilon-g(\lambda y-w)(f(Q+\varepsilon)-f(Q))+\lambda^2 W(\lambda y-w)\varepsilon\\
&=L_+\re\varepsilon+iL_-\im\varepsilon+\epsilon_1b^2|y|^2\varepsilon-\left(g(\lambda y-w)-1\right)df(Q)(\varepsilon)\\
&\hspace{20pt}-g(\lambda y-w)(f(Q+\varepsilon)-f(Q)-df(Q)(\varepsilon))+\lambda^2 W(\lambda y-w)\varepsilon,\\
i\frac{\partial \varepsilon}{\partial s}&=\frac{\partial H}{\partial \varepsilon}-\epsilon_1b^2|y|^2\varepsilon-(g(\lambda y-w)-1)f(Q)+\ModOp(Q+\varepsilon)+\Psi,
\end{align*}
where
\[
\ModOp v:=i\left(\frac{1}{\lambda}\frac{\partial \lambda}{\partial s}+b\right)\Lambda v-\left(1-\frac{\partial \gamma}{\partial s}\right)v-\left(\frac{\partial b}{\partial s}+b^2\right)\frac{|y|^2}{4}v+\left(\frac{1}{\lambda}\frac{\partial \lambda}{\partial s}+b\right)b\frac{|y|^2}{2}v-i\frac{1}{\lambda}\frac{\partial w}{\partial s}\cdot\nabla v-\frac{1}{2}\frac{b}{\lambda}\frac{\partial w}{\partial s}\cdot yv.
\]

For $\frac{\partial H}{\partial s}$, we have
\begin{align*}
\frac{\partial H}{\partial s}&=\epsilon_1b\frac{\partial b}{\partial s}\|y\varepsilon\|_2^2-\int_{\mathbb{R}^N}\left(\frac{\partial \lambda}{\partial s}y-\frac{\partial w}{\partial s}\right)\cdot(\nabla g)(\lambda y-w)\left(F(Q+\varepsilon)-F(Q)-dF(Q)(\varepsilon)\right)dy\\
&\hspace{20pt}+\lambda^2\frac{1}{\lambda}\frac{\partial\lambda}{\partial s}\int_{\mathbb{R}^N}W(\lambda y-w)|\varepsilon|^2dy+\frac{1}{2}\lambda^2\int_{\mathbb{R}^N}\left(\frac{\partial \lambda}{\partial s}y-\frac{\partial w}{\partial s}\right)\cdot(\nabla W)(\lambda y-w)|\varepsilon|^2dy.
\end{align*}
Therefore, we obtain
\begin{align}
\label{esti-1}
\frac{\partial H}{\partial s}\geq -\epsilon_1b^3\|y\varepsilon\|_2^2+o\left(b\left(\|\varepsilon\|_{H^1}+b^2\|y\varepsilon\|_2^2\right)\right).
\end{align}

For $\left\langle i\frac{\partial H}{\partial \varepsilon},i\frac{\partial \varepsilon}{\partial s}\right\rangle$, the following estimates hold:
\begin{align}
\label{esti-2}
\left|\left\langle i\frac{\partial H}{\partial \varepsilon},\epsilon_1b^2|y|^2\varepsilon\right\rangle\right|&\leq 2\epsilon_1b^2\|\varepsilon\|_{H^1}\|y\varepsilon\|_2+o(b\|\varepsilon\|_{H^1}^2),\\
\label{esti-3}
\left|\left\langle i\frac{\partial H}{\partial \varepsilon},(g(\lambda y-w)-1)f(Q)\right\rangle\right|&\lesssim s^{-(2+\kappa+L)},\\
\label{esti-4}
\left|\left\langle i\frac{\partial H}{\partial \varepsilon},\Psi\right\rangle\right|&\lesssim s^{-(2+\kappa+L)},\\
\label{esti-5}
\left|\left\langle i\frac{\partial H}{\partial \varepsilon},\ModOp Q\right\rangle\right|&\lesssim s^{-(4L-1)},\\
\label{esti-6}
\left|\left\langle i\frac{\partial H}{\partial \varepsilon},\ModOp \varepsilon\right\rangle\right|&\lesssim s^{-(4L-1)}.
\end{align}

Combining inequalities \eqref{esti-1}, \eqref{esti-2}, \eqref{esti-3}, \eqref{esti-4}, \eqref{esti-5}, and \eqref{esti-6}, we obtain
\begin{align*}
&\frac{d}{ds}H(s,\varepsilon(s))=\frac{\partial H}{\partial s}(s,\varepsilon(s))+\left\langle i\frac{\partial H}{\partial \varepsilon},i\frac{\partial \varepsilon}{\partial s}\right\rangle\\
\geq&-\epsilon_1b^3\|y\varepsilon\|_2^2+o\left(b\left(\|\varepsilon\|_{H^1}^2+b^2\|y\varepsilon\|_2^2\right)\right)-2\epsilon_1b^2\|\varepsilon\|_{H^1}\|y\varepsilon\|_2+o(b\|\varepsilon\|_{H^1}^2)-C\left(s^{-(2+\kappa+L)}+s^{-(4L-1)}\right)\\
\geq&-\epsilon_1b^3\|y\varepsilon\|_2^2-2\epsilon_1b^2\|\varepsilon\|_{H^1}\|y\varepsilon\|_2-\epsilon_4b\left(\|\varepsilon\|_{H^1}^2+b^2\|y\varepsilon\|_2^2\right)-bC\left(s^{-(2+\frac{3\kappa}{2})}+s^{-(2+\kappa)}\right)\\
\geq&-b\left(\left(\frac{\epsilon_1}{\epsilon_5}+\epsilon_4\right)\|\varepsilon\|_{H^1}^2+\left(1+\frac{\epsilon_4}{\epsilon_1}+\epsilon_5\right)\epsilon_1b^2\left\|y\varepsilon\right\|_2^2+Cs^{-(2+\kappa)}\right).
\end{align*}
\end{proof}

\begin{lemma}[Derivative of $S$ in time]
\label{Sdef}
For all $s\in(s_*,s_1]$,
\[
\frac{d}{ds}S(s,\varepsilon(s))\gtrsim \frac{b}{\lambda^m}\left(\|\varepsilon\|_{H^1}^2+b^2\left\|y\varepsilon\right\|_2^2-Cs^{-(2+\kappa)}\right).
\]
\end{lemma}

\begin{proof}
According to Lemma \ref{Hcoer}, Lemma \ref{Hdef}, and \eqref{modesti}, we have
\begin{align*}
\frac{d}{ds}S(s,\varepsilon(s))&=m\frac{b}{\lambda^m}H(s,\varepsilon(s))-m\frac{1}{\lambda^m}\left(\frac{1}{\lambda}\frac{\partial \lambda}{\partial s}+b\right)H(s,\varepsilon(s))+\frac{1}{\lambda^m}\frac{d}{ds}H(s,\varepsilon(s))\\
&\geq \frac{b}{\lambda^m}\left(\left(\frac{m\mu}{2}-\frac{\epsilon_1}{\epsilon_5}-\epsilon_2m-2\epsilon_3\right)\|\varepsilon\|_{H^1}^2+\epsilon_1\left(\frac{m}{2}-1-\frac{\epsilon_3m}{\epsilon_1}-\frac{2\epsilon_4}{\epsilon_1}-\epsilon_4\right)b^2\left\|y\varepsilon\right\|_2^2-Cs^{-(2+\kappa)}\right)
\end{align*}
From \eqref{epsidef},
\begin{align*}
\frac{m\mu}{2}-\frac{\epsilon_1}{\epsilon_5}-\epsilon_3m-2\epsilon_4&\geq\frac{m\mu}{2}-\frac{m\mu}{4}-\frac{m\mu}{24}-\frac{m\mu}{12}=\frac{m\mu}{8}\\
\frac{m}{2}-1-\frac{\epsilon_3m}{\epsilon_1}-\frac{2\epsilon_4}{\epsilon_1}-\epsilon_5&\geq \frac{\kappa}{4}-\frac{\kappa}{24\times 2}-\frac{\kappa}{24}-\frac{\kappa}{8}=\frac{\kappa}{16}
\end{align*}
hold.
\end{proof}

\section{Bootstrap}
\label{sec:bootstrap}
In this section, we establish the estimates of the decomposition parameters by using a bootstrap argument and the estimates obtained in Section \ref{sec:MEF}.

\begin{lemma}
\label{rebootstrap}
There exists a sufficiently small $\epsilon_3>0$ such that for all $s\in(s_*,s_1]$, 
\begin{align}
\label{reepsiesti}
\left\|\varepsilon(s)\right\|_{H^1}^2+b(s)^2\left\|y\varepsilon(s)\right\|_2^2&\lesssim s^{-\left(2L+\frac{\kappa}{2}\right)},\\
\label{relamesti}
\left|\frac{\lambda(s)}{\lambda_{\app}(s)}-1\right|+\left|\frac{b(s)}{b_{\app}(s)}-1\right|&\lesssim s^{-2(L-1)},\\
\label{rewesti}
\left|w(s)\right|&\lesssim s^{-(2L-1)}.
\end{align}
\end{lemma}

\begin{proof}
See \cite{LMR} for the proof of \eqref{reepsiesti}.

From Proposition \ref{profileprop} and \eqref{modesti},
\[
\left|E(Q_{\lambda,b,w,\gamma})-E_0\right|\leq\int_s^{s_1}\left|\left.\frac{d}{d\sigma}\right|_{\sigma=\tau}E(Q_{\lambda,b,w,\gamma}(\sigma))\right|d\tau\lesssim \int_s^{s_1}\tau^{-(1+\kappa)}d\tau\lesssim s^{-\kappa}
\]
holds. Therefore, since
\begin{align*}
\left|b^2\|yQ\|_2^2-8\lambda^2E_0\right|&\leq \lambda^2\left(\left|\frac{b^2}{\lambda^2}\|yQ\|_2^2-8E(P_{\lambda,b,\gamma})\right|+8\left|E_0-E(P_{\lambda,b,\gamma})\right|\right)\\
&\lesssim s^{-(2+\kappa)},
\end{align*}
we obtain
\begin{align*}
\left|\frac{\partial}{\partial s}\left(\sqrt{\frac{\|yQ\|_2^2}{8E_0}}\frac{1}{\lambda}-s\right)\right|&\leq\left|-\sqrt{\frac{\|yQ\|_2^2}{8E_0}}\frac{1}{\lambda^2}\frac{\partial \lambda}{\partial s}-1\right|\\
&\lesssim\frac{1}{\lambda}\left(\left|\frac{1}{\lambda}\frac{\partial\lambda}{\partial s}+b\right|+\left|b\|yQ\|-\sqrt{8E_0}\lambda\right|\right)\\
&\lesssim s^{-(2L-1)}+s^{-(1+\kappa)}.
\end{align*}
Since $\sqrt{\frac{\|yQ\|_2^2}{8E_0}}\frac{1}{\lambda(s_1)}=s_1$, we obtain
\[
\left|\sqrt{\frac{\|yQ\|_2^2}{8E_0}}\frac{1}{\lambda}-s\right|\lesssim s^{-2(L-1)},\quad\text{i.e., }\left|\frac{\lambda_{\text{app}}(s)}{\lambda(s)}-1\right|\lesssim s^{-(2L-1)}.
\]

Next, since
\[
\left|b^2-{b_{\text{app}}}^2\right|=\left|b^2-\frac{8E_0}{\|yQ\|_2^2}{\lambda_{\text{app}}}^2\right|\lesssim \left|b^2-\frac{8E_0}{\|yQ\|_2^2}{\lambda}^2\right|+\left|\lambda^2-{\lambda_{\text{app}}}^2\right|\lesssim s^{-(2+\kappa)}+s^{-2L},
\]
we obtain \eqref{relamesti}.

Finally, we prove $(\ref{rewesti})$. Since
\[
\left|w(s)\right|\leq \int_s^{s_1}|\text{Mod}(\sigma)|d\sigma\lesssim s^{-(2L-1)},
\]
$(\ref{rewesti})$ holds.
\end{proof}

From Lemma \ref{rebootstrap}, we obtain the following corollary:

\begin{corollary}
\label{reesti}
If $s_0$ is sufficiently large, then $s_*=s'=s_0$ for any $s_1>s_0$.
\end{corollary}

Finally, we rewrite the estimates obtained for the time variable $s$ in Lemma \ref{rebootstrap} into an estimates for the time variable $t$.

\begin{lemma}[Interval]
\label{interval}
If $s_0$ is sufficiently large, then there exists $t_0<0$ such that 
\[
[t_0,t_1]\subset {s_{t_1}}^{-1}([s_0,s_1]),\quad \left|\mathcal{C}s_{t_1}(t)^{-1}-|t|\right|\lesssim |t|^{1+\kappa}\ (t\in [t_0,t_1])
\]
hold for $t_1\in(t_0,0)$, where $\mathcal{C}=\frac{\|yQ\|_2^2}{8E_0}$.
\end{lemma}

\begin{proof}
See \cite{NP} for the proof.
\end{proof}

\begin{lemma}[Conversion of estimates]
\label{uniesti}
For $t\in[t_0,t_1]$, 
\begin{align*}
&\tilde{\lambda}_{t_1}(t)=\sqrt{\frac{8E_0}{\|yQ\|_2^2}}|t|\left(1+\epsilon_{\tilde{\lambda},t_1}(t)\right),\quad \tilde{b}_{t_1}(t)=\frac{8E_0}{\|yQ\|_2^2}|t|\left(1+\epsilon_{\tilde{b},t_1}(t)\right),\quad\left|\tilde{w}_{t_1}(t)\right|\lesssim |t|^{2L-1},\\
&\|\tilde{\varepsilon}_{t_1}(t)\|_{H^1}\lesssim |t|^{L+\frac{\kappa}{4}},\quad \|y\tilde{\varepsilon}_{t_1}(t)\|_{2}\lesssim |t|^{L+\frac{\kappa}{4}-1}
\end{align*}
holds. Furthermore,
\[
\sup_{t_1\in[t,0)}\left|\epsilon_{\tilde{\lambda},t_1}(t)\right|\lesssim |t|^\kappa,\quad \sup_{t_1\in[t,0)}\left|\epsilon_{\tilde{b},t_1}(t)\right|\lesssim |t|^\kappa.
\]
\end{lemma}

\section{Proof of Theorem \ref{theorem:EMBS}}
\label{sec:proof}
In this section, we prove Theorem \ref{theorem:EMBS}.

\begin{proof}[Proof of Theorem \ref{theorem:EMBS}]
Let $(t_n)_{n\in\mathbb{N}}\subset(t_0,0)$ be a monotonically increasing sequence such that $\lim_{n\nearrow \infty}t_n=0$. For each $n\in\mathbb{N}$, let $u_n$ be the solution for (NLS) with an initial value
\begin{align*}
u_n(t_n,x):=\frac{1}{{\lambda_{1,n}}^\frac{N}{2}}Q\left(\frac{x}{\lambda_{1,n}}\right)e^{-i\frac{b_{1,n}}{4}\frac{|x|^2}{{\lambda_{1,n}}^2}}
\end{align*}
at $t_n$, where
\[
s_n:=-\frac{\|yQ\|_2^2}{8E_0}{t_n}^{-1},\quad \lambda_n:=\sqrt{\frac{\|yQ\|_2^2}{8E_0}}{s_n}^{-1},\quad E(Q_{\lambda_n,b_n,0,0})=E_0.
\]

According to Lemma \ref{decomposition} with an initial value $\tilde{\gamma}_n(t_n)=0$, there exists a decomposition
\[
u_n(t,x)=\frac{1}{\tilde{\lambda}_n(t)^{\frac{N}{2}}}\left(Q+\tilde{\varepsilon}_n\right)\left(t,\frac{x+\tilde{w}_n(t)}{\tilde{\lambda}_n(t)}\right)e^{-i\frac{\tilde{b}_n(t)}{4}\frac{|x+\tilde{w}_n(t)|^2}{\tilde{\lambda}_n(t)^2}+i\tilde{\gamma}_n(t)}
\]
on $[t_0,t_n]$. Up to a subsequence, there exists $u_\infty(t_0)\in \Sigma^1$ such that
\[
u_n(t_0)\rightharpoonup u_\infty(t_0)\quad \text{weakly in}\ \Sigma^1,\quad u_n(t_0)\rightarrow u_\infty(t_0)\quad \text{in}\ L^2(\mathbb{R}^N)\quad (n\rightarrow\infty).
\]

Moreover, since $u_n:[t_0,0)\to\Sigma^1$ is locally uniformly bounded,
\[
u_n\rightarrow u_\infty\quad \text{in}\ C([t_0,T'],L^2(\mathbb{R}^N)),\quad u_n(t)\rightharpoonup u_\infty(t)\ \text{in}\ \Sigma^1 \quad (n\rightarrow\infty)
\]
holds (see \cite{NP}). Particularly, we have $\|u_\infty(t)\|_2=\|Q\|_2$.

Based on weak convergence in $H^1(\mathbb{R}^N)$ and Lemma \ref{decomposition}, we decompose $u_\infty$ to
\[
u_\infty(t,x)=\frac{1}{\tilde{\lambda}_\infty(t)^{\frac{N}{2}}}\left(Q+\tilde{\varepsilon}_\infty\right)\left(t,\frac{x+\tilde{w}_\infty(t)}{\tilde{\lambda}_\infty(t)}\right)e^{-i\frac{\tilde{b}_\infty(t)}{4}\frac{|x+\tilde{w}_\infty(t)|^2}{\tilde{\lambda}_\infty(t)^2}+i\tilde{\gamma}_\infty(t)}
\]
on $[t_0,0)$. Furthermore, as $n\rightarrow\infty$, 
\[
\tilde{\lambda}_n(t)\rightarrow\tilde{\lambda}_\infty(t),\quad \tilde{b}_n(t)\rightarrow \tilde{b}_\infty(t),\quad \tilde{w}_n(t)\rightarrow\tilde{w}_\infty(t),\quad e^{i\tilde{\gamma}_n(t)}\rightarrow e^{i\tilde{\gamma}_\infty(t)},\quad\tilde{\varepsilon}_n(t)\rightharpoonup \tilde{\varepsilon}_\infty(t)\quad \text{in}\ \Sigma^1
\]
holds for any $t\in[t_0,0)$. Therefore, we have
\begin{align*}
&\tilde{\lambda}_{\infty}(t)=\sqrt{\frac{8E_0}{\|yQ\|_2^2}}\left|t\right|(1+\epsilon_{\tilde{\lambda},0}(t)),\quad \tilde{b}_{\infty}(t)=\frac{8E_0}{\|yQ\|_2^2}\left|t\right|(1+\epsilon_{\tilde{b},0}(t)),\quad \left|\tilde{w}_\infty(t)\right|\lesssim |t|^{2L-1}\\
&\|\tilde{\varepsilon}_{\infty}(t)\|_{H^1}\lesssim \left|t\right|^{L+\frac{\kappa}{4}},\quad \|y\tilde{\varepsilon}_{\infty}(t)\|_2\lesssim \left|t\right|^{L+\frac{\kappa}{4}-1},\quad \left|\epsilon_{\tilde{\lambda},0}(t)\right|\lesssim |t|^\kappa,\quad \left|\epsilon_{\tilde{b},0}(t)\right|\lesssim |t|^\kappa
\end{align*}
from a uniform estimate of Lemma \ref{uniesti}. Consequently, we obtain Theorem \ref{theorem:EMBS}.

Finally, check the energy. Since $E'(w)=-\Delta w-g(x)|w|^\frac{4}{N}+Ww$, we obtain
\begin{align*}
E\left(u_n\right)-E\left(Q_{\tilde{\lambda}_n,\tilde{b}_n,\tilde{w}_n,\tilde{\gamma}_n}\right)=o_{t\nearrow0}(1),\quad E\left(u_\infty\right)-E\left(P_{\tilde{\lambda}_\infty,\tilde{b}_\infty,\tilde{w}_\infty,\tilde{\gamma}_\infty}\right)=o_{t\nearrow0}(1),
\end{align*}
where $o_{t\nearrow0}(1)$ is uniform with respect to $n$. From continuity of energy,
\[
\lim_{n\rightarrow \infty}E\left(Q_{\tilde{\lambda}_n,\tilde{b}_n,\tilde{w}_n,\tilde{\gamma}_n}\right)=E\left(P_{\tilde{\lambda}_\infty,\tilde{b}_\infty,\tilde{w}_\infty,\tilde{\gamma}_\infty}\right)
\]
holds and from conservation of energy,
\[
E\left(u_n\right)=E\left(u_n(t_n)\right)=E\left(P_{\tilde{\lambda}_{1,n},\tilde{b}_{1,n},0,0}\right)=E_0
\]
holds. Therefore, we obtain
\[
E\left(u_\infty\right)=E_0+o_{t\nearrow0}(1),
\]
so that $E\left(u_\infty\right)=E_0$.
\end{proof}

\section*{Acknowledgement}
The author would like to thank Masahito Ohta and Noriyoshi Fukaya for their support in writing this paper.

\end{document}